\documentclass{article}
\usepackage{amssymb,amsbsy}
\usepackage{amsmath}
\usepackage{amsthm}
\usepackage{mathrsfs}
\usepackage{color}
\usepackage[dvipsnames]{xcolor}
\usepackage{tikz}

\newtheorem{theorem}{Theorem}[section]
\newtheorem{lemma}[theorem]{Lemma}
\newtheorem{proposition}[theorem]{Proposition}
\newtheorem{corollary}[theorem]{Corollary}

\theoremstyle{definition}

\newtheorem{example}[theorem]{Example}

\theoremstyle{remark}
\newtheorem{remark}[theorem]{Remark}

\begin{document}
	
	\title{The extended Frobenius problem for Lucas series incremented by a Lucas number}
	
	\author{Aureliano M. Robles-P\'erez\thanks{Departamento de Matem\'atica Aplicada \& Instituto de Matem\'aticas (IMAG), Universidad de Granada, 18071-Granada, Spain. \newline E-mail: \textbf{arobles@ugr.es} (\textit{corresponding author}); ORCID: \textbf{0000-0003-2596-1249}.}
		\mbox{ and} Jos\'e Carlos Rosales\thanks{Departamento de \'Algebra \& Instituto de Matem\'aticas (IMAG), Universidad de Granada, 18071-Granada, Spain. \newline E-mail: \textbf{jrosales@ugr.es}; ORCID: \textbf{0000-0003-3353-4335}.} }
	
	\date{\today}
	
	\maketitle
	
	\begin{abstract}
		We study the extended Frobenius problem for sequences of the form $\{l_a\}\cup\{l_a+l_n\}_{n\in\mathbb{N}}$ and $\{l_a+l_n\}_{n\in\mathbb{N}}$, where $\{l_n\}_{n\in\mathbb{N}}$ is the Lucas series and $l_a$ is a Lucas number. As a consequence, we show that the families of numerical semigroups associated to both sequences satisfy the Wilf's conjecture.
	\end{abstract}
	
	\medskip
	
	\noindent{\it 2010 AMS Classification:} 11D07, 11B39 (Primary); 11A67, 05A17 (Secondary). 	
	
	\section{Introduction}
	
	Let $S\subseteq \mathbb{N}$ be the set generated by the sequence of positive integers $(a_1,\ldots,a_e)$, that is, $S=\langle a_1,\ldots,a_e \rangle = a_1{\mathbb N}+\cdots+a_e{\mathbb N}$. If $\gcd(a_1,\ldots,a_e)=1$, then it is well known that $S$ has a finite complement in $\mathbb{N}$. This fact leads to the classical problem in additive number theory called the Frobenius problem: what is the greatest integer $\mathrm{F}(S)$ which is not an element of $S$? Although this problem is solved for $e=2$ (see \cite{sylvester}), we have that it is not possible to find a polynomial formula to compute $\mathrm{F}(S)$ if $e\geq3$ (see \cite{curtis}). Therefore, many efforts have been made to obtain partial results or to develop algorithms to get the answer to this question (see \cite{alfonsin}).
	
	Another interesting question is to compute the cardinality $\mathrm{g}(S)$ of the set $\mathbb{N}\setminus S$. In fact, sometimes finding formulas for $\mathrm{F}(S)$ and $\mathrm{g}(S)$ is known as the extended Frobenius problem.
	
	Let us recall that the Lucas series (or sequence of Lucas numbers) is given by the recurrence relation $l_{n+2} = l_{n+1} + l_n$ for $n\geq0$ and the initial conditions $l_0=2,\, l_1=1$. This sequence was introduced by Lucas in \cite{lucas}.
	
	Among others, the main goal of this work is to solve the extended Frobenius problem for $S$ generated by Lucas series incremented by a Lucas number. That is, if $\{l_0,l_1,\ldots,l_n,\ldots\}$ is a Lucas series and $f_a$ is a Lucas number, then we will consider $S(a)=\langle l_a,l_a+l_0,l_a+l_1,\ldots,l_a+l_n,\ldots\rangle$. For instance, $S(0) = \langle 2,3 \rangle$, $S(1) = \langle 1 \rangle = \mathbb{N}$, $S(2) = \langle 3,4,5 \rangle$, $S(3) = \langle 4,5,6,7 \rangle$, and so on. Thus, our work can be considered along the lines of \cite{fel}. By the way, observe that in \cite{fel} the author always considers sequences of three numbers and we do not.	
	
	In order to achieve our purpose, we will use the theory of numerical semigroups (see Section~\ref{ns} for several results of this theory), which is closely related with the Frobenius problem. Indeed, the sets $S(a)$ defined above are numerical semigroups. In Section~\ref{msgS(a)} we will determine that, if $a\in\mathbb{N}\setminus\{0,1\}$, then $\{l_a,l_a+l_0,l_a+l_1,\ldots,l_a+l_{a-1}\}$ is the minimal finite subsequence that generates $S(a)$; applying that the Lucas series is a number system, in Section~\ref{aperyS(a)} we will explicitly give the Ap\'ery sets related to our numerical semigroups; and in Sections~\ref{frobeniusS(a)} and \ref{genusS(a)} we will give the formulas for solving the extended Frobenius problem; in addition, as a result derived from our study, we will check that our family of numerical semigroups satisfies the Wilf's conjecture (see \cite{wilf}). Finally, in Section~\ref{another}, we will analyse the numerical semigroups $T(a)$ generated by the sequences $\{l_a+l_0,l_a+l_1,\ldots,l_a+l_n,\ldots \}$; in particular, we will see that $S(a)=T(a) \cup \{l_a,2l_a+1\}$.

	\section{Preliminaries: numerical semigroups}\label{ns}
	
	Let $\mathbb{Z}$ be the set of integers and $\mathbb{N} = \{ z\in\mathbb{Z} \mid z\geq 0\}$. A submonoid of $(\mathbb{N},+)$ is a subset $M$ of $\mathbb{N}$ such that is closed under addition and contains de zero element. A \textit{numerical semigroup} is a submonoid of $(\mathbb{N},+)$ such that $\mathbb{N}\setminus S=\{n\in\mathbb{N} \mid n\not\in S\}$ is finite.
	
	Let $S$ be a numerical semigroup. From the finiteness of $\mathbb{N}\setminus S$, we can define two invariants of $S$. Namely, the \textit{Frobenius number of $S$} is the greatest integer that does not belong to $S$, denoted by $\mathrm{F}(S)$, and the \textit{genus of $S$} is the cardinality of $\mathbb{N}\setminus S$, denoted by $\mathrm{g}(S)$.
	
	If $X$ is a non-empty subset of $\mathbb{N}$, then we denote by $\langle X \rangle$ the submonoid of $(\mathbb{N},+)$ generated by $X$, that is,
	\[ \langle X \rangle=\big\{\lambda_1x_1+\cdots+\lambda_nx_n \mid n\in\mathbb{N}\setminus \{0\}, \ x_1,\ldots,x_n\in X, \ \lambda_1,\ldots,\lambda_n\in \mathbb{N}\big\}. \]
	It is well known (see Lemma~2.1 of \cite{springer}) that $\langle X \rangle$ is a numerical semigroup if and only if $\gcd(X)=1$.
	
	If $S$ is a numerical semigroup and $S=\langle X \rangle$, then we say that $X$ is a \textit{system of generators of $S$}. Moreover, if $S\not=\langle Y \rangle$ for any subset $Y\subsetneq X$, then we say that $X$ is a \textit{minimal system of generators of $S$}. In Theorem~2.7 of \cite{springer} it is shown that each numerical semigroup admits a unique minimal system of generators and that such a system is finite. We denote by $\mathrm{msg}(S)$ the minimal system of generators of $S$. The cardinality of $\mathrm{msg}(S)$, denoted by $\mathrm{e}(S)$, is the \textit{embedding dimension of $S$}.
	
	The (extended) Frobenius problem for a numerical semigroup $S$ consists of finding formulas that allow us to compute $\mathrm{F}(S)$ and $\mathrm{g}(S)$ in terms of $\mathrm{msg}(S)$. As in the case of the Frobenius problem for sequences, such formulas are well known for $\mathrm{e}(S)=2$ (see \cite{sylvester}), but it is not possible to find polynomial formulas when $e(S)\geq3$ (see \cite{curtis}), except for particular families of numerical semigroups.
	
	If $n\in S\setminus\{0\}$, then a very useful tool to describe a numerical semigroup $S$ is the set $\mathrm{Ap}(S,n)=\{s\in S \mid s-n\not\in S\}$, called the \textit{Ap\'ery set of $n$ in $S$} (after \cite{apery}). The following result is Lemma~2.4 of \cite{springer}.
	
	\begin{proposition}\label{prop1}
		Let $S$ be a numerical semigroup and $n\in S\setminus\{0\}$. Then the cardinality of $\mathrm{Ap}(S,n)$ is $n$. Moreover,
		\[ \mathrm{Ap}(S,n)=\{w(0)=0, w(1), \ldots, w(n-1)\}, \] 
		where $w(i)$ is the least element of $S$ congruent with $i$ modulo $n$.
	\end{proposition}
	
	The knowledge of $\mathrm{Ap}(S,n)$ allows us to solve the problem of membership of an integer to the numerical semigroup $S$. In fact, if $x\in\mathbb{Z}$, then $x\in S$ if and only if $x\geq w(x\bmod n)$. Moreover, we have the following result from \cite{selmer}.
	
	\begin{proposition}\label{prop2}
		Let $S$ be a numerical semigroup and let $n\in S\setminus\{0\}$. Then
		\begin{enumerate}
			\item $\mathrm{F}(S)=\max(\mathrm{Ap}(S,n))-n$,
			\item $\mathrm{g}(S)=\frac{1}{n}(\sum_{w\in \mathrm{Ap}(S,n)} w)-\frac{n-1}{2}$.
		\end{enumerate}
	\end{proposition}
	
	From this proposition, it is clear that, if we know an explicit description of $\mathrm{Ap}(S,n)$, then we have the solution of the Frobenius problem for $S$.

	\section{The minimal system of generators of $S(a)$}\label{msgS(a)}
	
	In this section, and unless otherwise indicated, we will assume that $a\in\mathbb{N}\setminus\{0,1\}$. Our main objective will be to determine the minimal system of generators of $S(a)=\langle l_a, l_a+l_0, l_a+l_1, \ldots, l_a+l_n \ldots\rangle$. First of all, let us observe that $\gcd\{l_a+l_0, l_a+l_1\} = \gcd\{l_a+2, l_a+1\} = 1$ and, therefore, $S(a)$ is a numerical semigroup.
	
	Furthermore, let us recall that the Fibonacci sequence is defined by the linear recurrence of order two $f_{n+2}=f_{n+1}+f_n$, for all $n\in\mathbb{N}$, with initial conditions $f_0=0$ and $f_1=1$. That is, $f_0=0$, $f_1=1$, $f_2=1$, $f_3=2$, $f_4=3$, $f_5=5$, $f_6=8$, and so on.
	
	\begin{lemma}\label{lem01}
		If $a\in\mathbb{N}\setminus\{0\}$ and $i\in\mathbb{N}$, then $l_{a+i} = f_{i+1}l_a + f_i l_{a-1}$.
	\end{lemma}
	
	\begin{proof}
		For $i=0$ we have that $l_{a} = l_a + 0\cdot l_{a-1} = f_1 l_a + f_0 l_{a-1}$, and for $i=1$ it is clear that $l_{a+1} = l_a + l_{a-1} = f_2 l_a + f_1 l_{a-1}$. By induction on $i$, let us now suppose that the result is true for all $j<i$. Then $l_{a+i} = l_{a+i-1} + l_{a+i-2} = (f_{i}l_a + f_{i-1} l_{a-1}) + (f_{i-1}l_a + f_{i-2} l_{a-1}) = (f_i + f_{i-1})l_a + (f_{i-1} + f_{i-2})l_{a-1} = f_{i+1}l_a + f_i l_{a-1}$.
	\end{proof}
	
	We are now ready to show the announced result on the minimal system of generators of $S(a)$.
	
	\begin{proposition}\label{prop02}
		If $a\in\mathbb{N}\setminus\{0,1\}$, then $\mathrm{msg}(S(a)) = \{ l_a, l_a+l_0, l_a+l_1, \ldots, l_a+l_{a-1}\}$.
	\end{proposition}
	
	\begin{proof}
		By Lemma~\ref{lem01}, if $i\in\mathbb{N}$, then $l_a + l_{a+i} = l_a + f_{i+1}l_a + f_i l_{a-1} =f_i(l_a+l_{a-1}) + (f_{i+1}-f_i+1)l_a \in \langle l_a, l_a+l_{a-1} \rangle$. Therefore, $\{ l_a, l_a+l_0, l_a+l_1, \ldots, l_a+l_{a-1}\}$ is a system of generators of $S(a)$. Since $l_a < l_a+l_1 < l_a+l_0 < l_a+l_2 < \ldots < l_a+l_{a-1} < 2l_a$, we easily conclude the proof. 
	\end{proof}
	
	As an immediate consequence of the previous proposition, we have the following result.
	
	\begin{corollary}\label{cor03}
		If $a\in\mathbb{N}\setminus\{0,1\}$, then the embedding dimension of $S(a)$ is $\mathrm{e}(S(a))=a+1$.
	\end{corollary}
	
	\begin{example}\label{exmp04}
		By definition, $S(6) = \langle 18, 18+2, 18+1, 18+3, 18+4, 18+7, 18+11, 18+18, 18+29, \ldots \rangle$. By Proposition~\ref{prop02}, we know that $\mathrm{msg}(S(6)) = \{18,19,20,21,22,25,29\}$ and, therefore, $\mathrm{e}(S(6))=7$.
	\end{example}

	\section{The Ap\'ery set of $S(a)$}\label{aperyS(a)}
	
	Our main objective in this section is to prove Theorem~\ref{thm08}, which describes $\mathrm{Ap}(S(a), l_a)$.
	
	It is a well known fact that every non-negative integer can be uniquely represented as a sum of non-consecutive Fibonacci numbers (see \cite{zeckendorf}), the so-called \textit{Zeckendorf decomposition}. For Lucas numbers, an analogue result is proved in \cite{brown}. Moreover, in \cite{decomposition} it is shown that, by considering the decompositions of an integer as a sum of Fibonacci numbers, the Zeckendorf decomposition is minimal in that no other decomposition has fewer summands. The reasoning behind this fact can be applied to decompositions of integers as sums of Lucas numbers.
	
	To work with a monotone sequence, from now on we will use a modification of the Lucas series. Namely, $\{\tilde{l}_n\}_{n\in\mathbb{N}}$ is the sequence given by $\tilde{l}_0=l_1$, $\tilde{l}_1=l_0$, and $\tilde{l}_n=l_n$ for all $n\geq2$.
	
	 From the above two paragraphs, we have the following result.
	
	\begin{lemma}\label{lem00}
		If $x\in\mathbb{N}\setminus\{0\}$, then there exists a unique $k\in\mathbb{N}$ such that $x=\sum_{i=0}^k b_i \tilde{l}_i$ with $(b_0,\ldots,b_k)\in\{0,1\}^{k+1}$, $b_k=1$, $b_0b_2=0$, and $b_ib_{i+1}=0$ for all $i\in\{0,\ldots,k-1\}$. Moreover, if $x=\sum_{i=0}^{k'} c_i \tilde{l}_i$ with $(c_0,\ldots,c_{k'})\in\mathbb{N}^{k'+1}$ and $k'\in\mathbb{N}$, then $\sum_{i=0}^k b_i \leq \sum_{i=0}^{k'} c_i$.
	\end{lemma}
	
	For each $x\in\mathbb{N}\setminus\{0\}$, the unique decomposition given in the previous result will be called the \textit{Zeckendorf-Lucas decomposition of $x$}.

	If $x\in\mathbb{N}$, then we denote by 
	\[ \beta(x) = \min\left\{ \sum_{i=0}^{k} b_i \;\bigg\vert\; x=\sum_{i=0}^{k} b_i \tilde{l}_i, \mbox{ with } (b_0,\ldots,b_k)\in\mathbb{N}^{k+1}, \; k\geq0 \right\}. \]
	
	\begin{remark}\label{rem00}
		By Lemma~\ref{lem00}, it is clear that, if $x=\sum_{i=0}^k b_i \tilde{l}_i$ is the Zeckendorf-Lucas decomposition of $x\in\mathbb{N}\setminus\{0\}$, then $\beta(x)=\sum_{i=0}^k b_i$. Moreover, $\beta(0)=0$.
	\end{remark}
	
	To prove Theorem~\ref{thm08} we need three prior lemmas.
	
	\begin{lemma}\label{lem05}
		Let $a\in\{1,2,3,4\}$ and $\sum_{i=0}^{a-1} b_i \tilde{l}_i \geq \tilde{l}_a$, with $(b_0,\ldots,b_{a-1})\in\mathbb{N}^{a}$. Then we have that $\sum_{i=0}^{a-1} b_i \tilde{l}_i - \tilde{l}_a = \sum_{i=0}^{a-1} c_i \tilde{l}_i$, with $(c_0,\ldots,c_{a-1})\in\mathbb{N}^{a}$ and $\sum_{i=0}^{a-1} c_i < \sum_{i=0}^{a-1} b_i$.
	\end{lemma}
	
	\begin{proof}
		If $a=1$, then $b_0\tilde{l}_0\geq\tilde{l}_1$ implies that $b_0\geq 2$ and we can take $c_0=b_2-2$ to ensure the result.
		
		For $a=2$ we have $b_0\tilde{l}_0 + b_1\tilde{l}_1\geq\tilde{l}_2$, that is, $b_0+2b_1\geq 3$. We study three cases.
		\begin{itemize}
			\item If $b_1=0$, then $b_0\geq 3$ and we have the result for $c_0=b_0-3$ and $c_1=0$.
			\item If $b_1=1$, then $b_0\geq 1$ and we can take $c_0=b_0-1$ and $c_1=0$.
			\item If $b_1\geq 2$, then it is sufficient to define $c_0=b_0+1$ and $c_1=b_1-2$.
		\end{itemize}
		
		For $a=3$, $b_0\tilde{l}_0 + b_1\tilde{l}_1 + b_2\tilde{l}_2 \geq\tilde{l}_3$ is equivalent to
		$b_0+2b_1+3b_2\geq 4$ and we see again three cases.
		\begin{itemize}
			\item If $b_2=0$, then $b_0+2b_1\geq 4$ and the proof is similar to the case $a=2$.
			\item If $b_2=1$, then $b_0+2b_1\geq 1$ (that is, $b_0+b_1\geq1$) and we can take $c_2=2$ and $c_0, c_1$ such that $c_0+c_1=b_0+b_1-1$.
			\item If $b_2\geq 2$, then it is enough to take $c_0=b_0$, $c_1=b_1+1$ and $c_2=b_2-2$.
		\end{itemize}
		
		Finally, for $a=4$ the proof follows the same ideas as for $a=3$.
	\end{proof}
	
	\begin{lemma}\label{lem06}
		Let $a\in\mathbb{N}\setminus\{0,1,2,3,4\}$ and $\sum_{i=0}^{a-1} b_i \tilde{l}_i \geq \tilde{l}_a$, with $(b_0,\ldots,b_{a-1})\in\mathbb{N}^{a}$. If ($b_{a-2}\geq 1$ and $b_{a-1}\geq 1$) or ($b_{a-2}=0$ and $b_{a-1}\geq 2$), then we have that $\sum_{i=0}^{a-1} b_i \tilde{l}_i - \tilde{l}_a = \sum_{i=0}^{a-1} c_i \tilde{l}_i$, with $(c_0,\ldots,c_{a-1})\in\mathbb{N}^{a}$ and $\sum_{i=0}^{a-1} c_i < \sum_{i=0}^{a-1} b_i$.
	\end{lemma}
	
	\begin{proof}
		Let us observe that, if $a\geq 5$, then 
		\[ \sum_{i=0}^{a-1} b_i \tilde{l}_i - \tilde{l}_a = \sum_{i=0}^{a-1} b_i \tilde{l}_i - \tilde{l}_{a-2} - \tilde{l}_{a-1} = \sum_{i=0}^{a-1} b_i \tilde{l}_i + \tilde{l}_{a-3} - 2\tilde{l}_{a-1}.\]
		
		Now, if $b_{a-2}\geq 1$ and $b_{a-1}\geq 1$, then $\sum_{i=0}^{a-1} b_i \tilde{l}_i - \tilde{l}_a = \sum_{i=0}^{a-1} c_i \tilde{l}_i$, with $c_i = b_i$ for $0\leq i\leq a-3$, $c_{a-2} = b_{a-2}-1$, and $c_{a-1} = b_{a-1}-1$. Thus, the result is proven in this case.
		
		Similarly, if $b_{a-2}=0$ and $b_{a-1}\geq 2$, then $\sum_{i=0}^{a-1} b_i \tilde{l}_i - \tilde{l}_a = \sum_{i=0}^{a-1} c_i \tilde{l}_i$, with $c_i = b_i$ for $0\leq i\leq a-4$, $c_{a-3} = b_{a-3}+1$, $c_{a-2} = b_{a-2} = 0$, and $c_{a-1} = b_{a-1}-2$. So, this case is also proven.
	\end{proof}
	
	\begin{lemma}\label{lem07}
		Let $a\in\mathbb{N}\setminus\{0,1\}$ and $\sum_{i=0}^{a-1} b_i \tilde{l}_i \geq \tilde{l}_a$, with $(b_0,\ldots,b_{a-1})\in\mathbb{N}^{a}$. Then there exists $(c_0,\ldots,c_{a-1})\in\mathbb{N}^{a}$ such that $\sum_{i=0}^{a-1} b_i \tilde{l}_i = \tilde{l}_a + \sum_{i=0}^{a-1} c_i l_i$ and $\sum_{i=0}^{a-1} c_i < \sum_{i=0}^{a-1} b_i$.
	\end{lemma}
	
	\begin{proof}
		We are going to prove the lemma using induction on $a$. 
		
		\textit{(Basis.)} By Lemma~\ref{lem05}, we have the result for $a\in\{2,3,4\}$.
		
		\textit{(Induction hypothesis.)} We now suppose that $a\geq 5$, $\sum_{i=0}^{a-1} b_i \tilde{l}_i \geq \tilde{l}_a$, and that the statement is true for all $k\in\{2,3,\ldots,a-1\}$.
		
		\textit{(Induction step.)} By Lemma~\ref{lem06}, if $a\geq 5$, then it is enough to study three cases. Let us recall that $\sum_{i=0}^{a-1} b_i \tilde{l}_i - \tilde{l}_a = \sum_{i=0}^{a-1} b_i \tilde{l}_i - \tilde{l}_{a-2} - \tilde{l}_{a-1}$.
		\begin{enumerate}
			\item If $b_{a-2}\geq 1$ and $b_{a-1}=0$, then $\sum_{i=0}^{a-1} b_i \tilde{l}_i - \tilde{l}_a = \sum_{i=0}^{a-2} b'_i \tilde{l}_i - \tilde{l}_{a-1}$, with $b'_i = b_i$ for $0\leq i\leq a-3$ and $b'_{a-2} = b_{a-2}-1$. Now, by the induction hypothesis for $k=a-1$, the result is proven in this case.
			
			\item If $b_{a-2}=0$ and $b_{a-1}=1$, then $\sum_{i=0}^{a-1} b_i \tilde{l}_i - \tilde{l}_a = \sum_{i=0}^{a-3} b_i \tilde{l}_i - \tilde{l}_{a-2}$. Then, by the induction hypothesis for $k=a-2$, the case is proven.
			
			\item If $b_{a-2}=b_{a-1}=0$, then $\sum_{i=0}^{a-1} b_i \tilde{l}_i - \tilde{l}_a = \sum_{i=2}^{a-3} b_i \tilde{l}_i - \tilde{l}_{a-2} - \tilde{l}_{a-1}$. Now, by the induction hypothesis for $k=a-2$ (observe that $\sum_{i=0}^{a-3} b_i \tilde{l}_i - \tilde{l}_{a-2} \geq \tilde{l}_{a-1} > 0$) and $k=a-1$, it follows that $\sum_{i=0}^{a-3} b_i \tilde{l}_i - \tilde{l}_{a-2} - \tilde{l}_{a-1} = \sum_{i=0}^{a-2} b'_i \tilde{l}_i - \tilde{l}_{a-1} = \sum_{i=0}^{a-1} c_i \tilde{l}_i$, with $b'_{a-2}=c_{a-1}=0$ and $\sum_{i=0}^{a-1} c_i = \sum_{i=0}^{a-2} c_i < \sum_{i=0}^{a-2} b'_i = \sum_{i=0}^{a-3} b'_i < \sum_{i=0}^{a-3} b_i= \sum_{i=0}^{a-1} b_i$. Therefore, the case is proven.
		\end{enumerate}
	\end{proof}
	
	We are now ready to prove the announced theorem. Let us recall that $l_a=\tilde{l}_a$ for all $a\in\mathbb{N}\setminus\{0,1\}$.
	
	\begin{theorem}\label{thm08}
		Let $a\in\mathbb{N}\setminus\{0,1\}$. If $x\in\{0,1,\ldots,l_a-1\}$ and $\mathrm{Ap}(S(a),l_a) = \{w(0)=0, w(1),\ldots,w(l_a-1)\}$, then $w(x) = \beta(x)l_a+x$.
	\end{theorem}
	
	\begin{proof}
		The result is trivial for $x=0$. So let us suppose that $x\in\{1,\ldots,l_a-1\}$.
		
		If $x=\sum_{i=0}^k b_i \tilde{l}_i$ is the Zeckendorf-Lucas decomposition of $x$, then $k<a$ and $\beta(x)l_a+x = \sum_{i=0}^{k} b_i (l_a+\tilde{l}_i) \in S(a)$. Moreover, $\beta(x)l_a+x \equiv x\pmod{l_a}$. Therefore, $w(x)\leq \beta(x)l_a+x$.
		
		We now suppose that $w(x)=\sum_{i=0}^{a-1} b'_i (l_a+\tilde{l}_i)$, with $(b'_0,\ldots,b'_{a-1})\in\mathbb{N}^{a}$. Obviously, $\sum_{i=0}^{a-1} b'_i \tilde{l}_i = x + \alpha l_a$ with $\alpha \in \mathbb{N}$. If $\alpha \geq 1$, then we can apply Lemma~\ref{lem07} and get that there exists $(c_0,\ldots,c_{a-1})\in\mathbb{N}^{a}$ such that $w(x)=\sum_{i=0}^{a-1} c_i (l_a+\tilde{l}_i) + l_a \left(1+\sum_{i=0}^{a-1} (b'_i-c_i)\right)$, with $\sum_{i=0}^{a-1} (b'_i-c_i)>0$. Therefore, $w(x)-l_a \in S(a)$, in contradiction with the fact that $w(x) \in\mathrm{Ap}(S(a),l_a)$. Thus, we have $\sum_{i=0}^{a-1} b'_i \tilde{l}_i=x$ and, in consequence, $w(x)=\left( \sum_{i=2}^{a-1} b'_i \right)l_a+x$. Finally, from the definition of $\beta(x)$, we can easily conclude that $w(x)\geq \beta(x)l_a+x$.
	\end{proof}
	
	\begin{example}\label{exmp-ap}
		By Example~\ref{exmp04}, we have $S(6)=\langle 18,19,20,21,22,25,29 \rangle$. Furthermore, from Theorem~\ref{thm08} and the corresponding Zeckendorf-Lucas decompositions, we deduce that
		\begin{itemize}
			\item $1=\tilde{l}_0;\, 2=\tilde{l}_1;\, 3=l_2;\, 4=l_3;\, 7=l_4;\, 11=l_5 \Rightarrow \beta(1)=\beta(2)=\beta(3)=\beta(4)=\beta(7)=\beta(11)=1 \Rightarrow w(1)=19;\, w(2)=20;\, w(3)=21;\, w(4)=22;\, w(7)=25;\, w(11)=29$;
			\item $5=l_3+\tilde{l}_0;\, 6=l_3+\tilde{l}_1;\, 8=l_4+\tilde{l}_0;\, 9=l_4+\tilde{l}_1;\, 10=l_4+l_2;\, 12=l_5+\tilde{l}_0;\, 13=l_5+\tilde{l}_1;\, 14=l_5+l_2;\, 15=l_5+l_3 \Rightarrow \beta(5)=\beta(6)=\beta(8)=\beta(9)=\beta(10)=\beta(12)=\beta(13)=\beta(14)=\beta(15)=2 \Rightarrow w(5)=41;\, w(6)=42;\, w(8)=44;\, w(9)=45;\, w(10)=46;\, w(12)=48;\, w(13)=49;\, w(14)=50;\, w(15)=51$;
			\item $16=l_5+l_3+\tilde{l}_0;\, 17=l_5+l_3+\tilde{l}_1 \Rightarrow \beta(16)=\beta(17)=3 \Rightarrow w(16)=70;\, w(17)=71$.
		\end{itemize}
	\end{example}

	\section{The Frobenius number of $S(a)$}\label{frobeniusS(a)}
	
	The main aim in this section is to prove Theorem~\ref{thm13}, which provides us a formula for the Frobenius number of $S(a)$ as a function of $a$ and $l_a$. For this we need some previous results.
	
	If $x\in\mathbb{N}$, then we denote by $\gamma(x) = \max\{ k\in\mathbb{N} \mid \tilde{l}_k\leq x \}$.
	
	\begin{remark}\label{rem01}
		By Lemma~\ref{lem00}, it is clear that, if $x=\sum_{i=0}^k b_i \tilde{l}_i$ is the Zeckendorf-Lucas decomposition of $x\in\mathbb{N}\setminus\{0\}$, then $\gamma(x)=k$. Moreover, if $x\geq3$, then $\gamma(x) = \max\{ k\in\mathbb{N} \mid l_k\leq x \}$.
	\end{remark}
	
	The following result is an immediate consequence of Remarks~\ref{rem00} and \ref{rem01} and the definitions of $\beta(x)$ and $\gamma(x)$.
	
	\begin{lemma}\label{lem09}
		If $x\in\mathbb{N}\setminus\{0\}$, then $\beta(x) = \beta\left(x-\tilde{l}_{\gamma(x)}\right) + 1$.
	\end{lemma}
	
	Since a Zeckendorf-Lucas decomposition does not admit consecutive Lucas numbers as addends, we easily have the following result.
	
	\begin{lemma}\label{lem09b}
		If $x\in\mathbb{N}\setminus\left\{ \{0\} \cup \{l_a\mid a\in\mathbb{N}\} \right\}$, then $\gamma\left(x-\tilde{l}_{\gamma(x)}\right) \leq \gamma(x) - 2$.
	\end{lemma}
	
	In some cases, we can very easily give $\beta(x)$. For example, $\beta(\tilde{l}_a)=1$ for all $a\in\mathbb{N}$. Let us see another case. As usual, $\lceil x \rceil = \min\{z\in\mathbb{Z} \mid x\leq z \}$.
	 
	\begin{lemma}\label{lem10}
		If $a\in\mathbb{N}\setminus\{1\}$, then $\beta(\tilde{l}_a-1) = \left\lceil \frac{a-1}{2} \right\rceil$.
	\end{lemma}
	
	\begin{proof}
		We will argue by mathematical induction on $a$. First, it is trivial that the result is true for $a\in\{0,2,3\}$.
		
		Now, by Lemma~\ref{lem09}, if $a\geq 4$, then $\beta(\tilde{l}_a - 1) = \beta(\tilde{l}_a - 1 - \tilde{l}_{a-1})+ 1 = \beta(\tilde{l}_{a-2} - 1) + 1$. Therefore, by the induction hypothesis on $a-2$, we have that $\beta(\tilde{l}_{a-2} - 1) = \left\lceil \frac{a-3}{2} \right\rceil$ and, consequently, $\beta(\tilde{l}_a - 1) = \left\lceil \frac{a-3}{2} \right\rceil + 1 = \left\lceil \frac{a-1}{2} \right\rceil$.
	\end{proof}
	
	In the general case, we can show an upper bound.
	
	\begin{lemma}\label{lem11}
		If $x\in\mathbb{N}\setminus\{l_a\mid a\in\mathbb{N}\}$, then $\beta(x) \leq \left\lceil \frac{\gamma(x)}{2} \right\rceil$.
	\end{lemma}
	
	\begin{proof}
		We will use mathematical induction on $x$. First, by simple calculations, the result is true for $x\in\{5,6\}$.
		
		Now, let us suppose that $x\geq 8$ and that $\beta(y) \leq \left\lceil \frac{\gamma(y)}{2} \right\rceil$ for all $y<x$ such that $y\not\in \{l_a\mid a\in\mathbb{N}\}$. Note that $x-\tilde{l}_{\gamma(x)}\not\in \{l_a\mid a\in\mathbb{N}\}$ because otherwise $x \in \{l_a\mid a\in\mathbb{N}\}$, contradicting the assumption of the lemma. Then, by Lemmas~\ref{lem09} and \ref{lem09b}, we have that
		\[ \beta(x) = \beta\left(x-\tilde{l}_{\gamma(x)}\right) + 1 \leq  \left\lceil \frac{\gamma\left(x-\tilde{l}_{\gamma(x)}\right)}{2} \right\rceil + 1 \leq \left\lceil \frac{\gamma(x)-2}{2} \right\rceil + 1 = \left\lceil \frac{\gamma(x)}{2} \right\rceil. \]
	\end{proof}
	
%

	We are ready to show the announced theorem.
	
	\begin{theorem}\label{thm13}
		If $a\in\mathbb{N}\setminus\{0,1\}$, then $\mathrm{F}(S(a)) =\left\lceil \frac{a-1}{2} \right\rceil l_a -1$.
	\end{theorem}
	
	\begin{proof}
		If $x\in\{0,1,\ldots,l_a-1\}$, then $\gamma(x)\leq a-1$. Therefore, by applying Theorem~\ref{thm08} and Lemmas~\ref{lem10} and \ref{lem11}, we easily deduce that $\max\left( \mathrm{Ap}(S(a),l_a)\right) = \left\lceil \frac{a-1}{2} \right\rceil l_a + l_a -1$. Finally, by applying Proposition~\ref{prop2}, we get that $\mathrm{F}(S(a)) = \left\lceil \frac{a-1}{2} \right\rceil l_a - 1$.
	\end{proof}
	
	\begin{example}\label{exmp14}
		By Example~\ref{exmp04}, we have that $S(6)=\langle 18,19,20,21,22,25,29 \rangle$. From Theorem~\ref{thm13}, we get that $\mathrm{F}(S(6)) = \left\lceil \frac{6-1}{2} \right\rceil l_6 - 1 = 53$.
	\end{example}
	
	Since $\mathrm{e}(S(a))=a+1$ and $\mathrm{m}(S(a))=l_a$, we can reformulate Theorem~\ref{thm13} as follows.
	
	\begin{corollary}\label{cor15}
		If $a\in\mathbb{N}\setminus\{0,1\}$, then $\mathrm{F}(S(a)) =\left\lceil \frac{\mathrm{e}(S(a))-2}{2} \right\rceil \mathrm{m}(S(a)) -1$.
	\end{corollary}

	\section{The genus of $S(a)$}\label{genusS(a)}
	
	In this section we will give a formula for the genus of $S(a)$. As usual, if $A$ is a set, then we denote by $\#(A)$ the cardinality of $A$. Moreover, if $m,n\in\mathbb{N}$ with $m\leq n$, then we denote by $\mathcal{F}_n(m)$ the set
	\[ \{ X\!\subseteq\! \{0,\ldots,n-1\} \mid  \#(X)=m \mbox{ and no two consecutive integers belong to } X \}. \]
	Moreover, for technical reasons, we take $\mathcal{F}_0(0)=\{\emptyset\}$.
	
	It is clear that $\#(\mathcal{F}_n(m))=0$ for all $m > \frac{n+1}{2}$. In other case, we have a classical result on counting subsets. 
	\begin{lemma}{\cite[Lemma~1]{kaplansky}}\label{lem16}
		If $m\in\mathbb{N}$, $n\in\mathbb{N}\setminus\{0\}$, and $m\leq \frac{n+1}{2}$, then $\#(\mathcal{F}_n(m)) = \binom{n+1-m}{m}$.
	\end{lemma}
	
	\begin{remark}\label{rem17}
	For $a\geq4$, the Zeckendorf-Lucas decomposition gives us a bijection between the sets $\{1,\ldots,l_a-1\}$ (recall that $l_a=\tilde{l}_a$ for all $a\in\mathbb{N}\setminus\{0,1\}$) and 
	\[ \mathcal{L}(a) = \left( \mathcal{F}_a(1) \cup \cdots \cup \mathcal{F}_a\left( \left\lfloor \frac{a+1}{2} \right\rfloor \right) \right) \setminus \left( \mathcal{F}_{a-4}(0) \cup \cdots \cup \mathcal{F}_{a-4}\left( \left\lfloor \frac{a-3}{2} \right\rfloor \right) \right). \]
	Indeed, if $x\in\{1,\ldots,l_a-1\}$ has the Zeckendorf-Lucas decomposition $\sum_{i=0}^k b_i \tilde{l}_i$ ($k<a$, $(b_0,\ldots,b_k)\in\{0,1\}^{k+1}$, $b_k=1$, $b_0b_2=0$, and $b_ib_{i+1}=0$ for all $i\in\{0,\ldots,k-1\}$), then we can associate $x$ with the set $B(x) \in \mathcal{F}_a(\beta(x)) \setminus \mathcal{F}_{a-4}(\beta(x)-2)$ consisting of all subscripts $j$ such that $b_j=1$ (let us observe that, with $\mathcal{F}_{a-4}(\beta(x)-2)$, we eliminate the sets in which $b_0=1, b_1=0, b_2=1, b_3=0$). Now, from the well known equality $f_a=\sum_{j=0}^{\left\lfloor \frac{a-1}{2} \right\rfloor} \binom{a-1-j}{j}$, the relation $l_a = f_{a+2}-f_{a-2}$, and the uniqueness of the Zeckendorf-Lucas decomposition, the correspondence associating $x$ to $B(x)$ is the sought bijection.	
	\end{remark}
	
	As a consequence of Theorem~\ref{thm08}, Lemma~\ref{lem16}, and Remark~\ref{rem17}, we have the following result.
	
	\begin{proposition}\label{prop18}
		If $a\in\mathbb{N}\setminus\{0,1,2,3\}$, then
		\[ \mathrm{Ap}\left(S(a),l_a\right)\setminus\{0\} = \left\{ \left(\#(B)\right)l_a + \sum_{b\in B} \tilde{l}_b \mid B \in \mathcal{L}(a)\setminus\{\emptyset\} \right\}. \]
		Moreover, if $\{B_1,B_2\} \subseteq \mathcal{L}(a)\setminus\{\emptyset\}$, then $\left(\#(B_1)\right)l_a + \sum_{b\in B_1} \tilde{l}_b = \left(\#(B_2)\right)l_a + \sum_{b\in B_2} \tilde{l}_b$ if and only if $B_1=B_2$.
	\end{proposition}
	
	The next result is an easy consequence of Proposition~\ref{prop2}.
	
	\begin{lemma}\label{lem19}
		If $S$ is a numerical semigroup, $n\in S\setminus\{0\}$, $\{k_1,k_2,\ldots,k_{n-1}\}\subseteq\mathbb{N}$, and $\mathrm{Ap}(S,n)=\{0,k_1n+1,k_2n+2,\ldots,k_{n-1}n+(n-1)\}$, then $\mathrm{g}(S)=k_1+k_2+\cdots+k_{n-1}$.
	\end{lemma}
	
	By Theorem~\ref{thm08} and Lemma~\ref{lem19}, we can deduce the following result.
	
	\begin{lemma}\label{lem20}
		If $a\in\mathbb{N}\setminus\{0,1\}$, then $\mathrm{g}(S(a)) = \sum_{x=1}^{l_a-1} \beta(x)$.
	\end{lemma}
	
	Let $B(x)$ be the set associated to $x\in\{1,\ldots,l_a-1\}$ in Remark~\ref{rem17}. Then it is clear that $\#(B(x)) = \beta(x)$. This fact, together Proposition~\ref{prop18} and Lemma~\ref{lem20}, leads to the next result, which is true for $a=3$ by direct substitution. 
	
	\begin{proposition}\label{prop21}
		If $a\in\mathbb{N}\setminus\{0,1,2\}$, then $\mathrm{g}(S(a)) = \sum_{i=1}^{\left\lfloor \frac{a+1}{2} \right\rfloor} i\binom{a+1-i}{i} - \sum_{i=0}^{\left\lfloor \frac{a-3}{2} \right\rfloor} (i+2)\binom{a-3-i}{i}$.
	\end{proposition}

	In fact, we can explicitly compute $\mathrm{g}(S(a))$ in the above proposition.
	
	\begin{theorem}\label{thm22}
		If $a\in\mathbb{N}\setminus\{0,1\}$, then $\mathrm{g}(S(a)) = \frac{a}{5} (l_{a} + l_{a-2})$.
	\end{theorem}
	
	\begin{proof}
		By direct substitution, we can check the result for $a\in\{2,3,4\}$. To prove it for $a\geq 5$, let us first see that
		\[ \mathrm{g}(S(a)) = \mathrm{g}(S(a-1)) + \mathrm{g}(S(a-2)) + l_{a-2}. \]
		
		For $k\in\mathbb{N}\setminus\{0\}$, let us take $a=2k+3$. Then, by Proposition~\ref{prop21}, we have that $\mathrm{g}(S(a)) = \mathrm{g}(S(2k+3)) = \sum_{i=1}^{k+2} i\binom{2k+4-i}{i} - \sum_{i=0}^{k} (i+2)\binom{2k-i}{i}$ and, hereafter,
		\[ \begin{split}
			\mathrm{g}(S(a)) = \, & \sum_{i=1}^{k+1} i \binom{2k+3-i}{i} + \sum_{i=1}^{k+1} i\binom{2k+3-i}{i-1} + (k+2)\binom{k+2}{k+2} \\
			& - \sum_{i=0}^{k-1} (i+2) \binom{2k-1-i}{i} - \sum_{i=1}^{k-1} (i+2) \binom{2k-1-i}{i-1} - (k+2)\binom{k}{k} \\
			= \, & \mathrm{g}(S(2k+2)) + \sum_{i=0}^{k+1} (i+1) \binom{2k+2-i}{i} - \sum_{i=0}^{k-1} (i+3) \binom{2k-2-i}{i} \\
			= \, & \mathrm{g}(S(2k+2)) + \mathrm{g}(S(2k+1)) + \sum_{i=0}^{k+1} \binom{2k+2-i}{i} - \sum_{i=0}^{k-1} \binom{2k-2-i}{i} \\
			= \, & \mathrm{g}(S(2k+2)) + \mathrm{g}(S(2k+1)) + f_{2k+3} - f_{2k-1} \\
			= \, & \mathrm{g}(S(a-1)) + \mathrm{g}(S(a-2)) + l_{a-2}.
		\end{split} \]
		
		If $a=2k+4$, with $k\in\mathbb{N}\setminus\{0\}$, the equality check is similar, so we omit it.
		
		To conclude that $\mathrm{g}(S(a)) = \frac{a}{5} (l_{a} + l_{a-2})$, we use mathematical induction. Let us take $a\geq5$ and assume that the equality is true for all $k\in\{3,4,\ldots,a-1\}$. Then,
		\[ \begin{split}
			\mathrm{g}(S(a)) & = \mathrm{g}(S(a-1)) + \mathrm{g}(S(a-2)) + l_{a-2} \\
			& = \frac{a-1}{5} \left(l_{a-1} + l_{a-3} \right) + \frac{a-2}{5} \left(l_{a-2} + l_{a-4} \right) + l_{a-2} \\
			& = \frac{a}{5} \left( l_{a-1} + l_{a-3} + l_{a-2} + l_{a-4} \right) + \frac{ 5l_{a-2} - l_{a-1} - l_{a-3} - 2l_{a-2} -2l_{a-4} }{5} \\
			& = \frac{a}{5} (l_{a} + l_{a-2}),
		\end{split} \]
		because $3l_{a-2} - l_{a-1} - l_{a-3} -2l_{a-4} = 0$.
	\end{proof}
	
	\begin{remark}
		If we take $l_{-1}=l_1-l_0=-1$, then Theorem~\ref{thm22} is true for $a=1$.
	\end{remark}
	
	\begin{example}\label{exmp23}
		By Example~\ref{exmp04}, we know that $S(6)=\langle 18,19,20,21,22,25,29 \rangle$. From Theorem~\ref{thm22}, we have that $\mathrm{g}(S(6)) = \frac{6}{5} (l_{6} + l_{4}) = 30$.
	\end{example}
	
	Since we know explicitly the expressions for the embedding dimension, the Frobenius number, and the genus of $S(a)$, we can check that this family of numerical semigroups satisfies the Wilf's conjecture (see \cite{wilf}). If $S$ is a numerical semigroup, then we denote by $\mathrm{n}(S)$ the cardinality of the set $\{s\in S \mid s<\mathrm{F}(S) \}$.
	
	\begin{corollary}\label{cor24}
		If $a\in\mathbb{N}$, then $\mathrm{F}(S(a)) + 1 \leq \mathrm{e}(S(a)) \mathrm{n}(S(a))$. 
	\end{corollary}
	
	\begin{proof}
		If $a\in\{0,1\}$, then $S(0) = \langle 2,3 \rangle$ and $S(1) = \mathbb{N}$. Therefore, the result is obvious in these cases.
		
		If $a\geq2$, we use an equivalent inequality. Indeed, since $\mathrm{g}(S) + \mathrm{n}(S) = \mathrm{F}(S) + 1$ for any numerical semigroup $S$, then
		\[ \mathrm{F}(S) + 1 \leq \mathrm{e}(S) \mathrm{n}(S) \Leftrightarrow \mathrm{e}(S) \mathrm{g}(S) \leq \left( \mathrm{e}(S)-1 \right) \left( \mathrm{F}(S)+1 \right). \]
		Now, by Corollary~\ref{cor03}, Theorem~\ref{thm13}, and Theorem~\ref{thm22}, we have that 
		\[ \mathrm{e}(S(a)) \mathrm{g}(S(a)) \leq \left( \mathrm{e}(S(a))-1 \right) \left( \mathrm{F}(S(a))+1 \right) \Leftrightarrow \]
		\[ (a+1) \frac{a(l_a+l_{a-2})}{5} \leq a \left\lceil \frac{a-1}{2} \right\rceil l_a. \]
		By direct verification, the last inequality is true for $a\in\{2,3,\ldots,8\}$. If $a\geq9$, since $l_{a-2}\leq l_a$, it suffices to see that $\frac{2}{5}(a+1)\leq \frac{1}{2}(a-1)$, which is obviously true.
	\end{proof}

	\section{Another family of numerical semigroups related to the Lucas series}\label{another}
	
	If $a\in\mathbb{N}$, then we denote by $T(a)$ the numerical semigroup generated by $\{l_a+l_n \mid n\in\mathbb{N}\}$. Thus, $T(0) = \langle 3,4,5 \rangle$, $T(1) = \langle 2,3 \rangle$, $T(2) = \langle 4,5,6,7 \rangle$, $T(3) = \langle 5,6,7,8 \rangle$, and so on.
	
	Our first purpose in this section will be to determine the minimal system of generators of $T(a)$.
	
	\begin{lemma}\label{lem30}
		If $\{ a,k \} \in \mathbb{N} \setminus \{0,1\}$, then $kl_a \in \langle l_a+l_0, l_a+l_1, \ldots, l_a+l_a \rangle$.
	\end{lemma}
	
	\begin{proof}
		Since $3l_a=(l_a+l_{a-1})+(l_a+l_{a-2})$, then it is clear that $\{2l_a,3l_a\} \in \langle l_a+l_0, l_a+l_1, \ldots, l_a+l_a \rangle$. Moreover, $kl_a \in \langle 2l_a, 3l_a \rangle$ for all $k \in \mathbb{N} \setminus \{0,1\}$. Thereby, we have the conclusion.
	\end{proof}
	
	\begin{proposition}\label{prop31}
		If $a \in\mathbb{N} \setminus \{0,1,2\}$, then $\mathrm{msg}(T(a))=\{l_a+l_0, l_a+l_1, \ldots, l_a+l_a\}$.
	\end{proposition}
	
	\begin{proof}
		Firstly, let us see that $l_a+l_{a+i} \in \langle l_a+l_0, l_a+l_1, \ldots, l_a+l_a \rangle$ for all $i\in\mathbb{N}\setminus\{0\}$. If $i=1$, then $l_a+l_{a+1} = (l_a+l_{a-2})+(l_a+l_{a-3})$ and the case is clear. Now, by Lemma~\ref{lem01}, if $i\geq 2$, then $l_a+l_{a+i} = l_a + f_{i+1}l_a + f_i l_{a-1} = (f_{i+1} - f_i + 1) l_a + f_i(l_a+l_{a-1})$. Since $f_{i+1} - f_i + 1\geq 2$ for all $i\geq 2$, we get the conclusion by applying Lema~\ref{lem30}.
		
		Finally, since $l_a+l_1 < l_a+l_0 < l_a+l_2 < \cdots < l_a+l_a < 2(l_a+l_1)$, the we easily deduce that $\{l_a+l_0, l_a+l_1, \ldots, l_a+l_a\}$ is the minimal system of generators of $T(a)$.
	\end{proof}
	
	As an immediate consequence of the above proposition we have the following result.
	
	\begin{corollary}\label{cor32}
		If $a \in\mathbb{N} \setminus \{0,1,2\}$, then $\mathrm{e}(T(a)) = a+1$.
	\end{corollary}
	
	Let us note that the structure of the Ap\'ery set of the numerical semigroups $T(a)$ with respect to its multiplicity (that is, the least element of $T(a)\setminus\{0\}$) is quite similar to the structure of the Ap\'ery set previously studied for the numerical semigroup $S(a)$. However, there are subtle differences that make it difficult to explicitly give such a set for $T(a)$. Nevertheless, we can avoid this obstacle by seeing that $T(a)\subsetneq S(a)$ and showing the elements of $S(a)\setminus T(a)$. For this we need the following lemma.
	
	\begin{lemma}\label{lem33}
		If $a \in\mathbb{N} \setminus \{0,1,2\}$ and $x\in T(a) \setminus \{0,l_a+l_1\}$, then $x+l_a\in T(a)$.
	\end{lemma}
	
	\begin{proof}
		By applying Proposition~\ref{prop31}, if $x\in T(a) \setminus \{0,l_a+l_1\}$, then we deduce that there exists $i\in\{0,2,\ldots,a\}$ such that $x=(l_a+l_i)+(x-l_a-l_i)$ and $x-l_a-l_i \in T(a)$. We distinguish two cases.
		\begin{enumerate}
			\item If $i=0$, then $(l_a+l_0)+l_a = (l_a+l_1) + (l_a+l_1) \in T(a)$ and, consequently, $x+l_a\in T(a)$.
			\item If $i\geq 2$, then $(l_a+l_i)+l_a = (l_a+l_{i-1}) + (l_a+l_{i-2}) \in T(a)$ and, therefore, $x+l_a\in T(a)$.
		\end{enumerate}
	\end{proof}
	
	\begin{proposition}\label{prop34}
		If $a\in\mathbb{N}\setminus\{0,1,2\}$, then $S(a)=T(a)\cup \{l_a,2l_a+1\}$. Moreover, $T(a)\cap \{l_a,2l_a+1\}=\emptyset$.
	\end{proposition}
	
	\begin{proof}
		It is clear that $T(a)\subseteq S(a)$, $\{l_a,2l_a+1\} \subseteq S(a)$, and $T(a)\cap \{l_a,2l_a+1\}=\emptyset$. Moreover, since all the elements of $\mathrm{Ap}(S(a),l_a)$ are non-negative integer linear combinations of $l_a+l_0, l_a+l_1, \ldots, l_a+l_{a-1}$, then we easily deduce that $\mathrm{Ap}(S(a),l_a)\subseteq T(a)$ from Propositions~\ref{prop02} and \ref{prop31}. Thus, to finish the proof, it suffices to see that $S(a)\setminus \{l_a,2l_a+1\}\subseteq T(a)$.
		
		If $x\in S(a)\setminus \{l_a,2l_a+1\}$, then there exist $k\in\mathbb{N}$ and $w\in\mathrm{Ap}(S(a),l_a)$ such that $x=kl_a+w$. We distinguish three cases.
		\begin{enumerate}
			\item If $k=0$, then $x=w\in \mathrm{Ap}(S(a),l_a)\subseteq T(a)$.
			\item If $k=1$, then $w\in \mathrm{Ap}(S(a),l_a)\setminus\{0,l_a+l_1\} \subseteq T(a)\setminus\{0,l_a+l_1\}$ and, by Lemma~\ref{lem33}, we conclude that $x=l_a+w\in T(a)$.
			\item If $k\geq 2$, then by Lemma~\ref{lem30} we have that $kl_a\in T(a)$. Thus, since $w\in \mathrm{Ap}(S(a),l_a) \subseteq T(a)$, we obtain that $x=kl_a+w\in T(a)$.
		\end{enumerate}
	\end{proof}
	
	As an immediate consequence of the above proposition we have the following result.
	
	\begin{corollary}\label{cor35}
		If $a\in\mathbb{N}\setminus\{0,1,2\}$, then $\mathrm{g}(T(a))=\mathrm{g}(S(a))+2=\frac{a}{5} (l_{a} + l_{a-2})+2$ and 
		\[\mathrm{F}(T(a))= \max\{\mathrm{F}(S(a)),2l_a+1\} = \left\{ \begin{array}{cl}
			2l_a+1, & \mbox{if }\, a\in\{3,4,5\}, \\[3pt] \left\lceil \frac{a-1}{2} \right\rceil l_a-1, & \mbox{if }\, a\geq 6.
		\end{array} \right. \]
	\end{corollary}
	
	\begin{remark}
		Following Corollary~\ref{cor24}, it is clear that $T(a)$ satisfies the Wilf's conjecture for all $a\in\mathbb{N}$.
	\end{remark}

	\section*{Acknowledgement}
	
	Both authors are supported by the project MTM2017-84890-P (funded by Mi\-nis\-terio de Econom\'{\i}a, Industria y Competitividad and Fondo Europeo de Desa\-rro\-llo Regional FEDER) and by the Junta de Andaluc\'{\i}a Grant Number FQM-343.

\end{document}